\documentclass[reqno]{amsart}

\usepackage{amssymb,amsmath,amsthm,amscd,latexsym,amsfonts}

\usepackage{xy}
\xyoption{all}
\usepackage{mathtools}
\usepackage[T1]{fontenc}
\usepackage{graphicx}
\usepackage{dsfont}
\usepackage{amsaddr}
\usepackage{cite}
\usepackage{bbm}%using the bbm fonts in math environment
\usepackage{color}
\usepackage{enumitem}

\usepackage[a4paper,top=3cm, bottom=3cm, left=3cm, right=3cm]{geometry}

\newtheorem{pr}{Proposition}

\newtheorem{lemma}{Lemma}
\newtheorem{de}{Definition}
\newtheorem{teo}{Theorem}

\begin{document}
\title[Leibniz algebras constructed by representations of General Diamond Lie algebras]
{Leibniz algebras constructed by representations of General Diamond Lie algebras}

\author{L.M.~Camacho\textsuperscript{1}, I.A.~Karimjanov\textsuperscript{2}, M.~Ladra\textsuperscript{3}, B.A.~Omirov\textsuperscript{4}}
\address{\textsuperscript{1}Dpto. Matem\'{a}tica Aplicada I.
Universidad de Sevilla. Avda. Reina Mercedes, 41012 Sevilla, lcamacho@us.es}
\address{\textsuperscript{2}Department of Algebra, University of Santiago de Compostela, 15782 Santiago de Compostela, Spain;   Department of Mathematics, Andijan State University, 170100 Andijan, Uzbekistan, iqboli@gmail.com}
\address{\textsuperscript{3}Department of Algebra, University of Santiago de Compostela, 15782 Santiago de Compostela, Spain, manuel.ladra@usc.es}
\address{\textsuperscript{4} Institute of Mathematics, National
University of Uzbekistan, 100125 Tashkent, omirovb@mail.ru}

\thanks{The work was partially supported  was supported by Ministerio de Econom\'ia y Competitividad (Spain),
grant MTM2013-43687-P (European FEDER support included) and  by Xunta de Galicia, grant GRC2013-045 (European FEDER support included).}

\begin{abstract}
In this paper we construct a minimal faithful representation of the $(2m+2)$-dimensional complex general Diamond Lie algebra, $\mathfrak{D}_m(\mathbb{C})$,
 which is isomorphic to a subalgebra of the special linear Lie algebra $\mathfrak{sl}(m+2,\mathbb{C})$.
We also construct a faithful representation of the general Diamond Lie algebra $\mathfrak{D}_m$  which is isomorphic to a subalgebra of the special symplectic Lie algebra $\mathfrak{sp}(2m+2,\mathbb{R})$.
Furthermore, we describe Leibniz algebras with corresponding $(2m+2)$-dimensional general Diamond Lie algebra $\mathfrak{D}_m$  and ideal generated by the squares of elements
 giving rise to a faithful representation of $\mathfrak{D}_m$.
\end{abstract}

\subjclass[2010]{17A32, 17B30, 17B10}
\keywords{Leibniz algebra, Diamond Lie algebra, Leibniz representation, right Lie module, classification}

\maketitle

\section{Introduction}

 Leibniz algebras
are a non-antisymmetric generalization of Lie algebras. They were introduced in 1965
by Bloh in \cite{Bloh}, who called them $D$-algebras, and in 1993   Loday  \cite{Lod} made them
popular and studied their (co)homology.

\begin{de} An algebra $(L,[-,-])$ over a field  $\mathbb{F}$   is called a Leibniz algebra if for any $x,y,z\in L$, the so-called Leibniz identity
\[ \big[x,[y,z]\big]=\big[[x,y],z\big]-\big[[x,z],y\big] \]
 holds.
\end{de}

  Since first works about Leibniz algebras around 1993 several researchers have tried to find analogs of important theorems in Lie algebras.
   For instance, the classical results on Cartan subalgebras \cite{AAO06,Omi}, Engel's theorem \cite{AyOm1}, Levi's decomposition \cite{Bar},
    properties of solvable algebras with given nilradical \cite{CLOK2} and others from the theory of Lie algebras are also true for Leibniz algebras.

Namely, an analogue of Levi's decomposition for Leibniz algebras asserts that any Leibniz algebra is decomposed into a semidirect sum of its solvable radical and a semisimple Lie algebra.
 Therefore, the main problem of the description of finite-dimensional Leibniz algebras consists of the study of solvable Leibniz algebras.

In fact,  each non-Lie Leibniz algebra $L$ contains a non-trivial ideal (later denoted by $I$), which is the subspace spanned by the squares of the elements of the algebra $L$.
  Moreover, it is easy to see that this ideal belongs to the right annihilator of $L$, that is $[L,I]=0$.
   Note also that the ideal $I$ is the minimal ideal with the property that the quotient algebra $L/I$ is a Lie algebra (the quotient algebra is said to be the corresponding Lie algebra to the Leibniz algebra $L$).

One of the approaches to the investigation of Leibniz algebras is a description of such algebras whose quotient algebra with respect to the ideal $I$ is a given Lie algebra \cite{ACKO15,CCO16,ORT13,UKO15}.

The map $I \times (L / I) \to I$ defined as $(v,\overline{x}) \mapsto [v,x]$, $v \in I, \, x \in L$,  endows $I$ with a structure of $(L /I)$-module.
If we consider the direct sum of vector spaces $Q(L) = (L / I) \oplus I$, then the operation $(-,-)$ defines a Leibniz algebra structure on $Q(L)$ with multiplication
\[[\overline{x},\overline{y}] = \overline{[x,y]}, \quad [\overline{x},v] = [x,v],
\quad [v, \overline{x}] = 0, \quad [v,w] = 0, \qquad x, y \in L, \ v,w \in I.\]

Therefore, for given a Lie algebra $G$ and a $G$-module $M$, we can construct a Leibniz algebra $L=G\oplus M$ by the above construction.

%In the paper \cite{Fil} the notion of irreducible Leibniz representation is introduced and it is shown that there are %only two kinds irreducible representations. Namely, one of then is coincide with Lie representation (it said to be Lie %representation) and another one has trivial action on the left side, while action on the right side is irreducible (it %said to be Leibniz representation). Following these concepts we shall say Leibniz representation (Leibniz module) for %that representation (respectively, module) which has trivial action on the left side.

The real general Diamond Lie algebra $\mathfrak{D}_m$ is a $(2m+2)$-dimensional Lie algebra with basis
\[\{J,P_1,P_2,\dots,P_m,Q_1,Q_2,\dots,Q_m,T\}\]
 and non-zero relations
\[
 [J,P_k]=Q_k, \qquad [J,Q_k]=-P_k, \qquad [P_k,Q_k]=T, \qquad 1\leq k\leq m.
\]

The complexification (for which we shall keep the same symbol $\mathfrak{D}_m(\mathbb{C})$) of the Diamond Lie
algebra is $\mathfrak{D}_m \otimes_\mathbb{R} \mathbb{C}$, and it shows the following (complex) basis:
\[P_k^{+} = P_k - iQ_k, \qquad Q_k^{-} = P_k + iQ_k, \qquad T, \qquad  J, \qquad 1\leq k\leq m,\]
where $i$ is the imaginary unit, and whose nonzero commutators are
\begin{equation} \label{eq3}
[J,P_k^{+}] = iP_k^{+}, \qquad [J,Q_k^{-}] = - iQ_k^{-}, \qquad [P_k^{+},Q_k^{-}] = 2iT,  \qquad 1\leq k\leq m.
\end{equation}

The Ado's theorem in Lie Theory states that every finite-dimensional complex Lie algebra can be represented as a matrix Lie algebra, formed by matrices.  However, that result does not specify which is the minimal order of the matrices involved in such representations. In \cite{Bur98}, the value
of the minimal order of the matrices for abelian Lie algebras  and Heisenberg algebras $\mathfrak{h}_m$,
 defined on a $(2m + 1)$-dimensional vector space with basis $X_1, \dots , X_m, Y_1, \dots Y_m, Z$,
and brackets $[X_i, Y_i] = Z$, is found. For  abelian Lie algebras of dimension $n$ the minimal order is $\lceil 2 \sqrt{n-1} \rceil$.

\begin{lemma}[\cite{Bur98}]\label{heis}
For the Heisenberg Lie algebras $\mathfrak{h}_m$, the minimal faithful matrix representation has order  equal  to $m+2$.
\end{lemma}

In this paper we find a minimal faithful representation of the $(2m+2)$-dimensional complex general Diamond Lie algebra, $\mathfrak{D}_m(\mathbb{C})$,
 which is isomorphic to a subalgebra of the special linear Lie algebra $\mathfrak{sl}(m+2,\mathbb{C})$.
  Moreover, we find a faithful representation of $\mathfrak{D}_m$ which is isomorphic to a subalgebra of the symplectic Lie algebra $\mathfrak{sp}(2m+2,\mathbb{R})$.
   Then we construct Leibniz algebras  with corresponding  general Diamond Lie algebra and the ideal generated by the squares of elements in these faithful representations.

\section{Leibniz algebras associated with minimal faithful representation of general Diamond Lie algebras}
%\section{A linear representation of general Diamond Lie algebras}

In this section we are going to study  Leibniz algebras $L$ such that $L/ I \cong \mathfrak{D}_m(\mathbb{C})$ and the
$\mathfrak{D}_m(\mathbb{C})$-module $I$ is a minimal faithful representation, that is,  the
action $I  \times \mathfrak{D}_m(\mathbb{C}) \to I$ gives rise to a minimal faithful representation of $\mathfrak{D}_m(\mathbb{C})$.
 Moreover, this representation factorizes through $\mathfrak{sl}(m+2,\mathbb{C})$.

\begin{pr}
Let $\mathfrak{D}_m(\mathbb{C})$ be a $(2m+2)$-dimensional general Diamond Lie algebra with basis
 \[\{J,P_1^+,P_2^+,\dots,P_m^+,Q_1^-,Q_2^-,\dots,Q_m^-,T\}.\]
 Then its minimal faithful
representation is given by
%Then it is isomorphic to a subalgebra of  $\mathfrak{sl}(m+2,\mathbb{C})$ by
\begin{multline*}
\theta J+\sum\limits_{k=1}^m\alpha_kP_k^{+}+\sum\limits_{k=1}^m\beta_kQ_k^{-}+\delta T
\mapsto  \\ \begin{pmatrix}
\frac{im}{m+2}\theta&\alpha_m&\alpha_{m-1}&\dots&\alpha_2&\alpha_1&-\frac{i}{2}\delta\\
0&-\frac{2i}{m+2}\theta&a_1&\dots&0&0&\beta_{m}\\
0&0&-\frac{2i}{m+2}\theta&\dots&0&0&\beta_{m-1}\\
\vdots&\vdots&\vdots&\ddots&\vdots&\vdots&\vdots\\
0&0&0&\dots&-\frac{2i}{m+2}\theta&a_1&\beta_2\\
0&0&0&\dots&0&-\frac{2i}{m+2}\theta&\beta_1\\
0&0&0&\dots&0&0&\frac{im}{m+2}\theta
\end{pmatrix}.
\end{multline*}
\end{pr}

\begin{proof}
Consider the bilinear map $\varphi \colon \mathfrak{D}_m(\mathbb{C})\rightarrow\mathfrak{sl}({m+2},\mathbb{C})$ given by
\begin{align*}
\varphi(J) &= \frac{im}{m+2}e_{1,1}-\sum\limits_{s=2}^{m+1}\frac{2i}{m+2}e_{s,s}+\frac{im}{m+2}e_{m+2,m+2},\qquad \varphi(T) = -\frac{i}{2}e_{1,m+2},\\
\varphi(P_k^+) & =e_{1,m+2-k}, \qquad \varphi(Q_k^-)=e_{m+2-k,m+2}, \qquad 1 \leq k \leq m,
\end{align*}
where $e_{i,j}$ is the matrix whose  ($i, j$)-th entry  is a $1$ and all others $0$'s.

By checking $[\varphi(x),\varphi(y)]=\varphi(x)\varphi(y)-\varphi(y)\varphi(x)$ for all  $x,y \in \mathfrak{D}_m(\mathbb{C})$, we verify that $\varphi$ is an injective homomorphism of algebras.  It is easy to see that $\mathfrak{D}_m \setminus J \cong \mathfrak{h}_m$. By Lemma~\ref{heis} we obtain that it is minimal.
\end{proof}

Let us denote by $V=\mathbb{C}^{m+2}$ the natural $\varphi(\mathfrak{D}_m(\mathbb{C}))$-module and endow it with a $\mathfrak{D}_m(\mathbb{C})$-module structure,
$V \times \mathfrak{D}_m(\mathbb{C}) \to V$, given by
 \[(x, e) \coloneqq x \varphi(e),\]
where $x \in V$ and $e\in\mathfrak{D}_m(\mathbb{C})$.

Then we obtain
\begin{equation}\label{eq6}\left\{\begin{array}{ll}
(X_1,J) = \frac{im}{m+2} X_1, & \\[1mm]
(X_k,J) = -\frac{2i}{m+2} X_k, & 2 \leq k \leq m+1, \\[1mm]
(X_{m+2},J) = \frac{im}{m+2} X_{m+2}, & \\[1mm]
(X_1, P_k^+) = X_{m+2-k}, & 1 \leq k \leq m,\\[1mm]
(X_{m+2-k},Q_{k}^-) = X_{m+2}, & 1 \leq k \leq m,\\[1mm]
(X_1,T) = -\frac{i}{2}X_{m+2}, &
\end{array}\right.\end{equation}
and the remaining products in the action being zero.

Now we investigate Leibniz algebras $L$ such that $L/I \cong \mathfrak{D}_m(\mathbb{C})$ and $I = V$ as a $\mathfrak{D}_m(\mathbb{C})$-module.

\begin{teo} Let $L$  be an arbitrary Leibniz algebra with corresponding Lie algebra $\mathfrak{D}_m(\mathbb{C})$ and $I$ associated with $\mathfrak{D}_m(\mathbb{C})$-module defined by \eqref{eq6}.
 Then there exists  a basis \[\{J,P_1^+,P_2^+,\dots,P_m^+,Q_1^-,Q_2^-,\dots,Q_m^-,T,X_1,X_2,\dots,X_{m+2}\}\] of $L$ such that \[[\mathfrak{D}_m(\mathbb{C}),\mathfrak{D}_m(\mathbb{C})]\subseteq\mathfrak{D}_m(\mathbb{C}).\]
\end{teo}

\begin{proof}
Here we shall use the multiplication table \eqref{eq3} of the complex Diamond Lie algebra.
Let us assume that \[[J,J]=\sum\limits_{k=1}^{m+2}\delta_kX_k.\] Then by
setting \[J' \coloneqq J+\frac{i(m+2)\delta_1}{m}X_1-\sum\limits_{k=2}^{m+1}\frac{i(m+2)\delta_i}{2}X_i+\frac{i(m+2)\delta_{m+2}}{m}X_{m+2},\] we can
assume that $[J,J]=0$.

Let us denote
\[[J,P_k^+]=iP_k^++\sum\limits_{s=1}^{m+2}\alpha_{k,s}X_s, \quad [J,Q_k^-]=-iQ_k^-+\sum\limits_{s=1}^{m+2}\beta_{k,s}X_s, \quad 1\leq k\leq m.\]

Taking the following basis transformation:
\[J^{\prime}=J, \quad  P_k^{+\prime}=P_k^+-\sum\limits_{s=1}^{m+2}i\alpha_{k,s}X_s, \quad Q_k^{-\prime}=Q_k^-+\sum\limits_{k=2}^{m+2}i\beta_{k,s}X_s, \quad T'=-i/2[P_1^{+\prime},Q_1^{-\prime}], \quad 1\leq k\leq m, \]
we can assume that
\[[J,P_k^+]=iP_k^+, \quad [J,Q_k^-]=-iQ_k^-, \quad [P_1^+,Q_1^-]=2iT, \quad 1\leq k\leq m.\]

By applying the  Leibniz identity to the triples $\{J,J,P_k^+\}, \ \{J,J,Q_k^-\}$, we derive
\[[P_k^+,J]=-[J,P_k^+], \qquad [Q_k^-,J]=-[J,Q_k^-], \qquad 1\leq k\leq m.\]

By considering Leibniz identity for the triples we have the following constraints.
\[\begin{array}{llll}
\text{ Leibniz identity }& & \text{ Constraints } &\\[1mm]
\hline \hline\\
\{P_1^+,J,Q_1^-\}&\Longrightarrow &[T,J]=0, &  \\[1mm]
\{J,T,J\}&\Longrightarrow &[J,T]=0, &  \\[1mm]
\{J,P_k^+,Q_s^-\}&\Longrightarrow &[P_k^+,Q_s^-]=-[Q_s^-,P_k^+], & 1\leq k,s\leq m,\\[1mm]
\{P_k^+,J,Q_s^-\}&\Longrightarrow &[P_k^+,Q_s^-]=0, & 1\leq k,s\leq m, \ k\neq s,\\[1mm]
\{P_k^+,J,Q_k^-\}&\Longrightarrow &[P_k^+,Q_k^-]=2iT, & 2\leq k\leq m, \\[1mm]
\{P_k^+,J,P_s^+\}&\Longrightarrow &[P_k^+,P_s^+]=0, & 1\leq k,s\leq m,\\[1mm]
\{Q_k^-,J,Q_s^-\}&\Longrightarrow &[Q_k^-,Q_s^-]=0, & 1\leq k,s\leq m,\\[1mm]
\{J,P_k^+,T\}&\Longrightarrow &[P_k^+,T]=0, & 1\leq k\leq m,  \\[1mm]
\{J,Q_k^-,T\}&\Longrightarrow & [Q_k^-,T]=0 & 1\leq k\leq m, \\[1mm]
\{P_k^+,P_k^+,Q_k^-\}&\Longrightarrow &[T,P_k^+]=0, & 1\leq k\leq m,  \\[1mm]
\{Q_k^-,Q_k^-,P_k^+\}&\Longrightarrow & [T,Q_k^-]=0, & 1\leq k\leq m. \\[1mm]
\end{array}\]
\end{proof}

\section{Leibniz algebras constructed by representation of general Diamond algebra which is isomorphic to a subalgebra of $\mathfrak{sp}(2m+2,\mathbb{R})$}

%\subsection{Leibniz algebras with the ideal $I$ as $\mathfrak{D}_m(\mathbb{R})$-modules by restriction of $\mathfrak{sp}(2m+2,\mathbb{R})$.}
In this section we are going to study  Leibniz algebras $L$ such that $L/ I \cong \mathfrak{D}_m$ and the
$\mathfrak{D}_m$-module $I$ is a faithful representation. Moreover, this representation factorizes through  $\mathfrak{sp}(2m+2,\mathbb{R})$.

\begin{pr}
Let $\mathfrak{D}_m$ be a $(2m+2)$-dimensional general Diamond Lie algebra with basis $\{J,P_1,P_2,\dots,P_m,Q_1,Q_2,\dots,Q_m,T\}$.
Then it is isomorphic to a subalgebra of $\mathfrak{sp}(2m+2,\mathbb{R})$ by
\begin{multline*}
a J+\sum\limits_{k=1}^mb_kP_k+\sum\limits_{k=1}^mc_kQ_k+d T \mapsto  \\
\left(\begin{array}{ccccc|ccccc}
0&b_1&b_2&\dots&b_m&c_m&\dots&c_2&c_1&2d\\
0&0&0&\dots&0&0&\dots&0&-a&c_1\\
0&0&0&\dots&0&0&\dots&-a&0&c_2\\
\vdots&\vdots&\vdots&\ddots&\vdots&\vdots&\vdots&\vdots&\vdots&\vdots\\
0&0&0&\dots&0&-a&\dots&0&0&c_m\\ \hline
0&0&0&\dots&a&0&\dots&0&0&-b_m\\
\vdots&\vdots&\vdots&\ddots&\vdots&\vdots&\ddots&\vdots&\vdots&\vdots\\
0&0&a&\dots&0&0&\dots&0&0&-b_2\\
0&a&0&\dots&0&0&\dots&0&0&-b_1\\
0&0&0&\dots&0&0&\dots&0&0&0
\end{array}\right).
\end{multline*}
\end{pr}

\begin{proof}
Consider the bilinear map $\varphi \colon \mathfrak{D}_m \to \mathfrak{sp}(2m+2,\mathbb{R})$ given by
\begin{align*}
\varphi(J) & = -\sum\limits_{s=2}^{m+1}e_{k,2m+3-k}+\sum\limits_{s=m+2}^{2m+1}e_{k,2m+3-k},\qquad \varphi(T) = 2e_{1,2m+2},\\
\varphi(P_k) & =e_{1,1+k}-e_{2m+2-k,2m+2}, \qquad \varphi(Q_k)=e_{1,2m+2-k}+e_{k+1,2m+2}, \qquad 1 \leq k \leq m.
\end{align*}

By checking $[\varphi(x),\varphi(y)]=\varphi(x)\varphi(y)-\varphi(y)\varphi(x)$ for all  $x,y \in \mathfrak{D}_m$, we verify that $\varphi$ is an injective hom+omorphism of algebras.
\end{proof}

Let us denote by $V=\mathbb{R}^{2m+2}$ the natural $\varphi(\mathfrak{D}_m)$-module and endow it with a $\mathfrak{D}_m$-module structure, $V \times \mathfrak{D}_m \to V$,
given  by  \[(x, e) \coloneqq  x \varphi(e),\]
where $x \in V$ and $e\in\mathfrak{D}_m$.

Then we obtain
\begin{equation}\label{eq7}\left\{\begin{array}{ll}
(X_k,J) = - X_{2m+3-k}, & 2 \leq k \leq m+1,\\[1mm]
(X_k,J) = X_{2m+3-k}, & m+2 \leq k \leq 2m+1,\\[1mm]
(X_1,P_k) = X_{k+1}, & 1 \leq k \leq m,\\[1mm]
(X_{2m+2-k},P_k) = - X_{2m+2}, & 1 \leq k \leq m,\\[1mm]
(X_1,Q_k) = X_{2m+2-k}, & 1 \leq k \leq m,\\[1mm]
(X_{k+1},Q_k) = X_{2m+2}, & 1 \leq k \leq m, \\[1mm]
(X_1,T) = 2X_{2m+2}, & 1 \leq k \leq m,\\[1mm]
\end{array}\right.\end{equation}
and the remaining products in the action being zero.

\begin{teo} An arbitrary Leibniz algebra with corresponding Lie algebra $\mathfrak{D}_m$ and $I$ associated with $\mathfrak{D}_m$-module defined by \eqref{eq7} admits a basis $\{J,P_1,P_2,\dots,P_m,Q_1,Q_2,\dots,Q_m,T,X_1,X_2,\dots,X_{2m+2}\}$ such that the  multiplication table $[\mathfrak{D}_m,\mathfrak{D}_m]$ has the following form:

\[ \left\{\begin{array}{ll}
[J,J]=a_1X_{2m+2}, & [J,P_k]=-[P_k,J]=Q_k,   \\[1mm]
[J,Q_k]=-[Q_k,J]=-P_k, & [P_k,Q_k]=-[Q_k,P_k]=T, \\[1mm]
[P_k,P_s]=[Q_k,Q_s]=b_{k,s}X_{2m+2}, & [P_k,Q_s]=[Q_k,P_s]=c_{k,s}X_{2m+2},
\end{array}\right.\]
with the restrictions
\[ b_{k,s}=-b_{s,k}, \qquad c_{k,s}=c_{s,k},\]
where $1\leq k,s\leq m, \ k\neq s$.
\end{teo}

\begin{proof}
Let us assume that \[[J,J]=\sum\limits_{k=1}^{m+2}\delta_kX_k.\] Then by
setting \[J^{\prime}=J+\sum\limits_{k=2}^{m+1}\delta_{2m+3-k}X_k-\sum\limits_{k=m+2}^{2m+1}\delta_{2m+3-k}X_k,\] we can
assume that \[[J,J]=\delta_1X_1+\delta_{2m+2}X_{2m+2}.\]

Let us suppose that $[J,T]=\sum\limits_{k=1}^{2m+2}\rho_kX_k$ and considering the Leibniz identity to $\{J,T,J\}$, we get
\[\delta_1=0, \quad [J,T]=\rho_1X_1+\rho_{2m+2}X_{2m+2}.\]

By making the change of basis element $J'=J-\frac{1}{2}\rho_{2m+2}X_1$
we get the \[[J,T]=\rho_1X_1.\]

Let us suppose
\[[J,P_k]=Q_k+\sum\limits_{s=1}^{2m+2}\lambda_{k,s}X_s, \quad [J,Q_k]=-P_k+\sum\limits_{s=1}^{2m+2}\mu_{k,s}X_s, \quad 1\leq k\leq m.\]

Taking the following basis transformation:
 \[J^{\prime}=J, \quad  P_k^{\prime}=P_k-\sum\limits_{s=1}^{2m+2}\mu_{k,s}X_s, \quad Q_k^{\prime}=Q_k+\sum\limits_{k=1}^{2m+2}\lambda_{k,s}X_s, \quad T^{\prime}=[P_1^{\prime},Q_1^{\prime}], \quad 1\leq k\leq m, \]
we can assume that
\[[J,P_k]=Q_k, \qquad [J,Q_k]=-P_k, \qquad [P_1,Q_1]=T, \qquad 1\leq k\leq m.\]

By applying the  Leibniz identity to the triples $\{J,J,P_k\}, \ \{J,J,Q_k\}$ we derive
\[[P_k,J]=-[J,P_k], \qquad [Q_k,J]=-[J,Q_k], \qquad 1\leq k\leq m.\]

By verifying the Leibniz identity on elements, we have the following the restrictions.
\[\begin{array}{llll}
\text{ Leibniz identity }& & \text{ Constraints } &\\[1mm]
\hline \hline\\
\{J,P_k,T\}&\Longrightarrow &[Q_k,T]=\rho_1X_{k+1}, & 1\leq k\leq m, \\[1mm]
\{J,Q_k,T\}&\Longrightarrow &[P_k,T]=-\rho_1X_{2m+2-k}, & 1\leq k\leq m, \\[1mm]
\{P_1,T,Q_1\}&\Longrightarrow &[T,T]=0, &  \\[1mm]
\end{array}\]

We set
\[\left\{\begin{array}{ll}
[P_j,Q_j]=T+\sum\limits_{t=1}^{2m+2}\beta_{j,t}X_t,  &[Q_k,P_k]=-T+\sum\limits_{t=1}^{2m+2}\gamma_{k,t}X_t, \\[1mm]
[P_k,P_s]=\sum\limits_{t=1}^{2m+2}\eta_{k,s,t}X_t,  &[Q_k,Q_s]=\sum\limits_{t=1}^{2m+2}\theta_{k,s,t}X_t, \\[1mm]
[P_k,Q_s]=\sum\limits_{t=1}^{2m+2}\nu_{k,s,t}X_t, \ k\neq s,  &[Q_k,P_s]=\sum\limits_{t=1}^{2m+2}\xi_{k,s,t}X_t, \ k\neq s,
\end{array}\right.\]
where  \ $ 2\leq j\leq m, \ 1\leq k,s\leq m$.

By applying the  Leibniz identity  to $\{P_k,P_s,T\}$ and $\{Q_k,Q_s,T\}$, we obtain
\[\eta_{k,k,1}=\theta_{k,k,1}=\frac{1}{2}\rho_1, \qquad \eta_{k,s,1}=\theta_{k,s,1}=0, \qquad  1\leq k,s\leq m, \ k\neq s.  \]

By considering the next equality
\begin{align*}
\rho_1X_1 & =[J,T]=\big[J,[P_k,Q_k]\big]=\big[[J,P_k],Q_k\big]-\big[[J,Q_k],P_k\big]\\
{} &= [Q_k,Q_k]+[P_k,P_k]=\rho_1X_1+\sum\limits_{s=2}^{2m+2}(\eta_{k,k,s}+\theta_{k,k,s})X_s,
\end{align*}
we get \[ \theta_{k,k,s}=-\eta_{k,k,s}, \qquad 2\leq s\leq 2m+2, \ 1\leq k\leq m.\]

Analogously, by applying the  Leibniz identity  to $\{J,P_k,Q_s\},\{J,Q_k,Q_s\}$ and $\{J,P_k,P_s\}$, we get
\begin{equation}\label{eq8}[P_k,P_s]=-[Q_s,Q_k], \quad [P_k,Q_s]=[P_s,Q_k], \quad [Q_k,P_s]=[Q_s,P_k], \quad 1\leq k,s\leq m, \ k\neq s.\end{equation}

By applying the  Leibniz identity to $\{P_1,J,Q_1\}$ and $\{P_1,P_1,Q_1\}$,
we have
\[[T,J]=\sum\limits_{s=2}^{2m+2}2\eta_{1,1,s}X_s, \qquad [T,P_1]=\frac{3}{2}\rho_1X_{2m+1}+\eta_{1,1,2}X_{2m+2}.\]

By the next identity
\begin{align*}
[Q_1,[J,P_1]]& =[[Q_1,J],P_1]-[[Q_1,P_1],J]=[P_1,P_1]-[-T+\sum\limits_{s=1}^{2m+2}\gamma_{1,s}X_s,J]\\
{} &=\frac{1}{2}\rho_1X_1+\sum\limits_{s=1}^{2m+2}\eta_{1,1,s}X_s+\sum\limits_{s=2}^{2m+2}2\eta_{1,1,s}X_s+
\sum\limits_{s=2}^{m+1}\gamma_{1,s}X_{2m+3-s}-\sum\limits_{s=m+2}^{2m+1}\gamma_{1,s}X_{2m+3-s}\\
{} &=\frac{1}{2}\rho_1X_1+\sum\limits_{s=1}^{2m+2}3\eta_{1,1,s}X_s+\sum\limits_{s=2}^{m+1}\gamma_{1,s}X_{2m+3-s}-\sum\limits_{s=m+2}^{2m+1}\gamma_{1,s}X_{2m+3-s}.
\end{align*}
On the other hand
\[[Q_1,[J,P_1]]=[Q_1,Q_1]=\frac{1}{2}\rho_1X_1-\sum\limits_{s=1}^{2m+2}\eta_{1,1,s}X_s,\]
and from this we  deduce
\[\gamma_{1,s}=-4\eta_{1,1,2m+3-s}, \qquad \gamma_{1,k}=4\eta_{1,1,2m+3-k},  \qquad \eta_{1,1,2m+2}=0,\]
with $ 2\leq s\leq m+1, \quad  m+2\leq k\leq 2m+1$.

Now by considering the identity
\begin{align*}
T& =[P_1,Q_1]=[P_1,[J,P_1]]=[[P_1,J],P_1]-[[P_1,P_1],J]\\
{} &=-[Q_1,P_1] -[\frac{1}{2}\rho_1X_1+\sum\limits_{k=2}^{2m+1}\eta_{1,1,k}X_k,J]\\
{} &=T-\gamma_{1,1}X_1+\sum\limits_{k=2}^{m+1}4\eta_{1,1,2m+3-k}X_k-\sum\limits_{k=m+2}^{2m+1}4\eta_{1,1,2m+3-k}X_k\\
{}& \quad +\gamma_{1,2m+2}X_{2m+2}+\sum\limits_{k=2}^{m+1}\eta_{1,1,k}X_{2m+3-k}-\sum\limits_{k=m+2}^{2m+1}\eta_{1,1,k}X_{2m+3-k},
\end{align*}
we get \[\gamma_{1,1}=\gamma_{1,2m+2}=\eta_{1,1,k}=0, \quad 2\leq k\leq 2m+1.\]

By the next  Leibniz identity
\begin{align*}
\rho_1X_2 & =[Q_1,T]=[Q_1,[P_1,Q_1]]=[[Q_1,P_1],Q_1]-[[Q_1,Q_1],P_1]\\
{}&=-[T,Q_1]-[\frac{1}{2}\rho_1X_1,P_1]=-[T,Q_1]-\frac{1}{2}\rho_1X_2,
\end{align*}
we obtain \[[T,Q_1]=-\frac{3}{2}\rho_1X_2.\]

Hence, we have \[\left\{\begin{array}{ll}
[T,J]=0,  &[Q_1,P_1]=-T, \\[1mm]
[P_1,P_1]=\frac{1}{2}\rho_1X_1, &[Q_1,Q_1]=\frac{1}{2}\rho_1X_1, \\[1mm]
[T,P_1]=\frac{3}{2}\rho_1X_{2m+1},  &[T,Q_1]=-\frac{3}{2}\rho_1X_{2},
\end{array}\right.\]

By using the next Leibniz identity
\begin{align*}
-\rho_1X_{2m+2-k}&=[P_k,T]=[P_k,[P_k,Q_k]]=[[P_k,P_k],Q_k]-[[P_k,Q_k],P_k]\\
{}& =[\frac{1}{2}\rho_1X_1+\sum\limits_{s=2}^{2m+2}\eta_{k,k,s}X_s,Q_k]-[T+\sum\limits_{s=1}^{2m+2}\beta_{k,s}X_s,P_k] \\
{}&=\frac{1}{2}\rho_1X_{2m+2-k}+\eta_{k,k,k+1}X_{2m+2} -[T,P_k]-\beta_{k,1}X_{k+1}+\beta_{k,2m+2-k}X_{2m+2},
\end{align*}
we get
\[[T,P_k]=\frac{3}{2}\rho_1X_{2m+2-k}-\beta_{k,1}X_{k+1}+(\beta_{k,2m+2-k}+\eta_{k,k,k+1})X_{2m+2}, \quad 2\leq k\leq m.\]

By applying the Leibniz identity to the elements $\{P_k,J,Q_k\}$ and $\{Q_k,J,P_k\}$,  we get
\[\beta_{k,s}=\gamma_{k,s}=-2\eta_{k,k,2m+3-s}, \qquad \beta_{k,t}=\gamma_{k,t}=2\eta_{k,k,2m+3-t}, \qquad \eta_{k,k,2m+2}=0,\]
where  $2\leq k\leq m, \quad 2\leq s\leq m+1, \quad m+2\leq t\leq 2m+1$.

By the next Leibniz identity applied to $\{P_k,J,P_k\}$, we have
\[\gamma_{k,1}=-\beta_{k,1}, \qquad \eta_{k,k,s}=0, \qquad 2\leq k\leq m, \quad 2\leq s\leq 2m+1.\]

Now, by considering
\begin{align*}
\rho_1X_{k+1}&=[Q_k,T]=[Q_k,[P_k,Q_k]]=[[Q_k,P_k],Q_k]-[[Q_k,Q_k],P_k]\\
{}&=[-T-\beta_{k,1}X_1,Q_k]
-[\frac{1}{2}\rho_1X_1,P_k]=-[T,Q_k]-\beta_{k,1}X_{2m+2-k}-\frac{1}{2}\rho_1X_{k+1},
\end{align*}
 we get
\[[T,Q_k]=-\frac{3}{2}\rho_1X_{k+1}-\beta_{k,1}X_{2m+2-k}.\]

By using the Leibniz identity for $\{T,P_k,Q_k\}$,
we get \[\beta_{k,1}=0, \qquad 2\leq k\leq m.\]

So, we have
\[\left\{\begin{array}{l}
[P_k,Q_k]=-[Q_k,P_k]=T, \\[1mm]
[P_k,P_k]=[Q_k,Q_k]=\frac{1}{2}\rho_1X_1, \\[1mm]
[T,P_k]=\frac{3}{2}\rho_1X_{2m+2-k}, \\[1mm]
[T,Q_k]=-\frac{3}{2}\rho_1X_{k+1},
\end{array}\right.\]
where $2\leq k\leq m$.

By verifying Leibniz identity on elements, we obtain the following  restrictions.
\[\begin{array}{llll}
\text{ Leibniz identity }& & \text{ Constraints } &\\[1mm]
\hline \hline\\
\{P_k,P_s,T\}&\Longrightarrow &\eta_{k,s,1}=0, & 1\leq k,s\leq m, \ k\neq s, \\[1mm]
\{Q_k,Q_s,T\}&\Longrightarrow &\theta_{k,s,1}=0, & 1\leq k,s\leq m, \ k\neq s, \\[1mm]
\{P_k,Q_s,T\}&\Longrightarrow &\nu_{k,s,1}=0, & 1\leq k,s\leq m, \ k\neq s,  \\[1mm]
\{Q_k,P_s,T\}&\Longrightarrow &\xi_{k,s,1}=0, & 1\leq k,s\leq m, \ k\neq s.
\end{array}\]

By applying the Leibniz identity  to $\{P_k,P_s,J\}, \ \{Q_k,Q_s,J\}$,  we get
\[[[P_k,P_s],J]=-[Q_k,P_s]-[P_k,Q_s], \qquad [[Q_k,Q_s],J]=[Q_k,P_s]+[P_k,Q_s],\]
it follows that
\[[[P_k,P_s],J]=-[[Q_k,Q_s],J],\]
hence \[\xi_{k,s,2m+2}=-\nu_{k,s,2m+2}, \quad \theta_{k,s,t}=-\eta_{k,s,t}, \quad 1\leq k,s\leq m, \ 2\leq t\leq 2m+1, \ k\neq s.\]
and \[\begin{array}{ll}
\nu_{k,s,t}+\xi_{k,s,t}=-\eta_{k,s,2m+3-t}&2\leq t\leq m+1,\\
\nu_{k,s,t}+\xi_{k,s,t}=\eta_{k,s,2m+3-t}&m+2\leq t\leq 2m+1.
\end{array}\]

Let us consider the identity
\[ [[Q_k,P_s],J]=[Q_k,[P_s,J]]+[[Q_k,J],P_s]=-[Q_k,Q_s]+[P_k,P_s]\]

We have that $\theta_{k,s,2m+2}=\eta_{k,s,2m+2}$ and
\[\begin{array}{ll}
\xi_{k,s,t}=-2\eta_{k,s,2m+3-t},& \qquad 2\leq t\leq m+1,\\
\xi_{k,s,t}=2\eta_{k,s,2m+3-t},& \qquad m+2\leq t\leq m+1,
\end{array}\]
and
\[\begin{array}{ll}
\nu_{k,s,t}=\eta_{k,s,2m+3-t}, & \qquad 2\leq t\leq m+1,\\
\nu_{k,s,t}=-\eta_{k,s,2m+3-t}, & \qquad m+2\leq t\leq 2m+1.
\end{array}\]

%\]-\sum\limits_{t=2}^{2m+1}\eta_{k,s,t}X_t-\theta_{k,s,2m+2}X_{2m+2}+\sum\limits_{t=2}^{2m+2}\eta_{k,s,t}X_t=(\eta_{k,s,2m+2}-\theta_{k,s,2m+2})X_{2m+2},\]
%On other hand
%\][[Q_k,P_s],J]=[\sum\limits_{t=1}^{2m+2}\xi_{k,s,t}X_{t},J]=-\sum\limits_{t=2}^{m+1}\xi_{k,s,t}X_{2m+3-t}+\sum\limits_{t=m+2}^{2m+1}\xi_{k,s,t}X_{2m+3-t}\]
%Follows we have
%\begin{equation}\label{eq14}\theta_{k,s,2m+2}=\eta_{k,s,2m+2}, \quad \xi_{k,s,t}=0,  \quad 1\leq k,s\leq m, \ 2\leq t\leq 2m+1, \ k\neq s.\end{equation}

Analogously,by applying the Leibniz identity to $\{P_k,Q_s,J\}$, we get
\[\begin{array}{ll}
\nu_{k,s,t}=-2\eta_{k,s,2m+3-t},& \qquad 2\leq t\leq m+1,\\
\nu_{k,s,t}=2\eta_{k,s,2m+3-t},& \qquad m+2\leq t\leq 2m+1.
\end{array}\]

We get that $\nu_{k,s,t}=0,  \ 1\leq k,s\leq m, \ 2\leq t\leq 2m+1, \ k\neq s$. It implies that $\eta_{k,s,t}=\xi_{k,s,t}=\theta_{k,s,t}=0$ for $1\leq k,s\leq m, \ 2\leq t\leq 2m+1, \ k\neq s$.

Hence, we have
\[\begin{array}{ll}
[P_k,P_s]=[Q_k,Q_s]=\eta_{k,s,2m+2}X_{2m+2}, & \qquad  1\leq k,s\leq m, \quad k\neq s, \\[1mm]
[P_k,Q_s]=[Q_k,P_s]=\nu_{k,s,2m+2}X_{2m+2}, &  \qquad 1\leq k,s\leq m, \quad k\neq s.
\end{array}\]

By equation \eqref{eq8} we have the following restrictions \[\eta_{k,s,2m+2}=-\eta_{s,k,2m+2}, \qquad \nu_{k,s,2m+2}=\nu_{s,k,2m+2}, \qquad 1\leq k,s\leq m, \ k\neq s.\]

Finally, we apply the Leibniz identity to the elements $\{P_k,P_k,P_s\}$ with $k\neq s$ and we obtain $\rho_1=0$.
We denote again $(\delta_{2m+2},\eta_{k,s,2m+2},\nu_{k,s,2m+2})=(a_1,b_{k,s},c_{k,s})$.
\end{proof}

%\bibliographystyle{elsart-num-sort}
%\bibliography{biblio_Iqbol}

\end{document}